\documentclass[11pt]{article}
\usepackage{amsfonts,amsmath,amssymb,amsthm, color, enumerate,graphicx}
\usepackage{dsfont}

\usepackage[numbers]{natbib}
\def\fullpage {
\addtolength{\topmargin}{-2 cm}
\addtolength{\oddsidemargin}{-1.6 cm} \addtolength{\textwidth}{+3.0 cm}
\addtolength{\textheight}{+3.3 cm}}
\fullpage
\newcommand{\sfrac}[2]{\scriptscriptstyle{\frac{#1}{#2}}}

\newtheorem{theorem}{Theorem}
\newtheorem{lemma}[theorem]{Lemma}

\newtheorem{corollary}[theorem]{Corollary}

\newtheorem{question}{Question}
\newtheorem{problem}[question]{Problem}
\newtheorem{proposition}[theorem]{Proposition}
\newtheorem*{definition}{Definition}

\begin{document}

\title{Full subgraphs}

\author{
{\large{Victor Falgas-Ravry}}\thanks{Institutionen f\"or matematik och matematisk statistik, Ume{\aa}  Universitet, 901 87 Ume{\aa}, Sweden. \footnotesize{Email: {\tt victor.falgas-ravry@umu.se}. Research supported by fellowships from the Kempe foundation and the Mittag-Leffler Institute.}}
\and
{\large{Klas Markstr\"{o}m}}\thanks{Institutionen f\"or matematik och matematisk statistik, Ume{\aa}  Universitet, 901 87 Ume{\aa}, Sweden. Research supported by a grant from Vetenskapsr{\aa}det.
\footnotesize {E-mail: {\tt klas.markstrom@umu.se}.}}
\and{\large{Jacques Verstra\"ete}}\thanks{Department of Mathematics, University of California at San Diego, 9500
Gilman Drive, La Jolla, California 92093-0112, USA. \footnotesize{ E-mail: {\tt jverstra@math.ucsd.edu.} Research supported by NSF Grant DMS-110 1489. }}}

\date{}
\maketitle


\begin{abstract}
Let $G=(V,E)$ be a graph of density $p$ on $n$ vertices. Following Erd\H{o}s, \L uczak and Spencer, an $m$-vertex subgraph $H$ of $G$ is called {\em full} if
$H$ has minimum degree at least $p(m - 1)$. Let $f(G)$ denote the order of a largest full subgraph of $G$.
If $p\binom{n}{2}$ is a non-negative integer, define
\[ f(n,p) = \min\{f(G) : \vert V(G)\vert = n, \ \vert E(G)\vert = p\textstyle{{n \choose 2}}\}.\]
Erd\H{o}s, \L uczak and Spencer proved that for $n \geq
2$,
\[ (2n)^{\frac{1}{2}} - 2 \leq f(n, {\sfrac{1}{2}}) \leq 4n^{\sfrac{2}{3}}(\log n)^{\sfrac{1}{3}}.\]
In this paper, we prove the following lower bound: for $n^{-\sfrac{2}{3}} <p_n <1-n^{-\sfrac{1}{7}}$,
\[ f(n,p) \geq \frac{1}{4}(1-p)^{\sfrac{2}{3}}n^{\sfrac{2}{3}} -
1.\] Furthermore we show that this is tight up to a multiplicative constant factor for infinitely many $p$ near
the elements of
$\{\frac{1}{2},\frac{2}{3},\frac{3}{4},\dots\}$.
In contrast, we show that for any $n$-vertex graph $G$,
either $G$ or $G^c$ contains a full subgraph on
$\Omega(\frac{n}{\log n})$ vertices.
Finally, we discuss full subgraphs of random and pseudo-random graphs, and several open problems.
\end{abstract}

\section{Introduction}
\subsection{Full subgraphs}
A {\em full subgraph} of a graph $G$ of density $p$ is an $m$-vertex subgraph $H$ of minimum degree at least $p(m - 1)$. This notion was introduced by Erd\H{o}s, \L uczak and Spencer~\cite{ELS}. We may think of a full subgraph as a particularly `rich' subgraph, with `unusually high' minimum degree: if we select $m$ vertices of $G$ uniformly at random, then the expected \emph{average} degree of the subgraph they induce is exactly $p(m-1)$, which is the minimum degree we require for the subgraph to be full. One cannot in general expect to find $m$-vertex subgraphs of higher minimum degree, as may be seen for example by considering complete multipartite graphs with parts of equal sizes. Fixing $p$ and $n$, one may ask for the largest $m$ such that every $n$-vertex graph $G$ has a full subgraph with $m$ vertices.
For a graph $G$, let $f(G)$ denote the largest number of vertices in a full subgraph of $G$. If $p{n \choose 2}$ is a non-negative integer, define
\[ f(n,p) = \min\{f(G) : \vert V(G)\vert = n, \vert E(G)\vert = p\textstyle{{n \choose 2}}\}.\]
Erd\H{o}s, \L uczak and Spencer~\cite{ELS} raised a
problem  which, in our terminology, amounts to
determining $f(n,p)$ when $p = \frac{1}{2}$, and showed
$(2n)^{\frac{1}{2}} - 2 \leq f(n, {\sfrac{1}{2}}) \leq (2
+ \tfrac{2}{\sqrt{3}}) n^{\sfrac{2}{3}}(\log
n)^{\sfrac{1}{3}}$.
In this paper, we prove the following theorem for general $p$, improving the lower bound of~\cite{ELS}:
\begin{theorem} \label{main}
For all $p = p_n$ such that $n^{-\sfrac{2}{3}} <p_n <1-n^{-\sfrac{1}{7}}$,
\[ f(n,p) \geq  \frac{1}{4}(1-p)^{\sfrac{2}{3}}n^{\sfrac{2}{3}}-1.\]
Moreover for each $c \geq 1$, if $p = \frac{r}{r + 1} +
cn^{-\sfrac{2}{3}}$ for some $r \in \mathbb N$, then
$f(n,p) = \Theta(n^{\sfrac{2}{3}})$.
\end{theorem}
We also show in Section~\ref{proofofgreedy} (after the
proof of Theorem~\ref{greedy}) that if $p \leq
n^{-\sfrac{2}{3}}$ then $\vert f(n,p) -
p^{\sfrac{1}{2}}n\vert \leq 1$.
A case of particular interest is $p = \frac{1}{2}$, where Theorem \ref{main} together with the results of Erd\H{o}s, \L uczak and Spencer~\cite{ELS} gives
\[ \frac{1}{4\sqrt[3]{4}}n^{\sfrac{2}{3}} -1\leq f(n,\tfrac{1}{2}) \leq \left(2 + \frac{2}{\sqrt{3}}\right)n^{\sfrac{2}{3}}(\log n)^{\sfrac{1}{3}}.\]
A similar construction to that of Erd{\H o}s, {\L}uczak
and Spencer shows that $f(n,p)$ is within a logarithmic
factor of $n^{\sfrac{2}{3}}$ when $p \in
\{\frac{1}{2},\frac{2}{3},\frac{3}{4},\dots\}$. The order
of magnitude of $f(n,p)$ is not known in general, and we
pose the following problem:

\begin{problem}
For each fixed $p \in (0,1)$, determine the order of
magnitude of $f(n,p)$.
\end{problem}

It may also be interesting to determine the order of
magnitude of the minimum possible value of $f(G)$ when
$G$ is from a certain class of $n$-vertex graph of
density $p$, such as $K_r$-free graphs or graphs of
chromatic number at most $r$.

\subsection{Discrepancy and full subgraphs}
Full subgraphs are related to subgraphs of large positive
discrepancy. For a graph $G$ of density $p$ and an
$m$-vertex set $X \subseteq V(G)$, let $\delta(X) = e(X) -
p{m \choose 2}$, where $e(X)$ denotes the number of edges
of $G$ that lie in $X$. The {\em positive} and {\em
negative discrepancy} of $G$ are respectively defined by
\begin{eqnarray*}
\mbox{disc}^+(G) = \max_{X \subseteq V(G)} \delta(X) \quad \quad \mbox{disc}^-(G) = \max_{X \subseteq V(G)}(-\delta(X)).
\end{eqnarray*}
The {\em discrepancy} of $G$ is $\mbox{disc}(G) = \max\{\mbox{disc}^+(G),\mbox{disc}^-(G)\}$.
We prove the following simple bound via a greedy algorithm in Section \ref{proofofgreedy}, relating the positive discrepancy to the order of a largest full subgraph:
\begin{theorem}\label{greedy}
Let $G$ be a graph of density $p$ with $\mbox{disc}^+(G) = \alpha>0$. Then
\[ f(G) \geq (1 - p)^{-\sfrac{1}{2}} (2\alpha)^{\sfrac{1}{2}}.\]
\end{theorem}

Theorem~\ref{greedy} is best possible, since the graph $G$ consisting of a clique with ${m \choose 2}$ edges
and $n-m$ isolated vertices has $f(G) = m$ and
\[ \mathrm{disc}^+(G)= \binom{m}{2}\left(1-\tfrac{\binom{m}{2}}{\binom{n}{2}}\right),\]
whereas the bound given by Theorem~\ref{greedy} is $f(G)
\geq \left\lceil \sqrt{m(m - 1)}\right\rceil = m$. On the other
hand, if $G$ is any $n$-vertex graph obtained by adding
or removing $o(n^{\frac{4}{3}})$ edges in a complete
multipartite graph with a bounded number of parts of
equal size, then $\mbox{disc}^+(G) = o(n^{\frac{4}{3}})$
and the lower bound in Theorem \ref{greedy} is superseded
by Theorem \ref{main}.

That there should be a relation between full subgraphs
(which have unexpectedly high minimum degree) and
subgraphs with large positive discrepancy (which have
unexpectedly many edges) is not surprising. Indeed, an easy observation is that any subgraph maximising the positive
discrepancy must be a full subgraph (see
Lemma~\ref{obs: max-disc implies full/co-full}).


\subsection{Random and pseudo-random graphs}
In a random or pseudo-random setting, we are able to improve our bounds on the size of a largest full subgraph by drawing on previous work on discrepancy and jumbledness. Jumbledness was introduced in a seminal paper of Thomason~\cite{Thomason87} as a measure of the `pseudo-randomness' of a graph.
\begin{definition}\label{def: jumbled}
A graph $G$ is $(p, j)$-jumbled if for every $X\subseteq
V(G)$, $\vert \delta_p(X) \vert \leq j \vert X\vert$.
\end{definition}
We prove that graphs which are `well-jumbled' --- meaning that they are $(p,j)$-jumbled for some small $j$, and so look `random-like' --- have large full subgraphs.
\begin{theorem}\label{theorem: f(G) geq disc+/alpha}
Suppose $G$ is a $(p,j)$-jumbled graph of density $p$.
Then
\[f(G)\geq \frac{\mathrm{disc}^+(G)}{j}.\]
\end{theorem}
This result, which for small values of $j$ improves on Theorems~\ref{main} and~\ref{greedy}, is proved in Section~\ref{section: jumbledness}.

For random graphs, Erd\H{o}s, \L uczak and Spencer~\cite{ELS} showed that for $p=\frac{1}{2}$, $f(G_{n,p}) \geq \beta_1 n - o(n)$ asymptotically almost surely, where $\beta_1 \approx 0.227$. Using Theorem~\ref{theorem: f(G) geq disc+/alpha}, we can extend this linear lower bound to arbitrary, fixed $p\in(0,1)$. Erd\H{o}s and Spencer~\cite{ES} proved that for $p=\frac{1}{2}$ we have asymptotically almost surely
\begin{equation}\label{discrandom}
\mbox{disc}^+(G_{n,p}) = \Theta\bigl(p^{\sfrac{1}{2}}(1 - p)^{\sfrac{1}{2}} n^{\sfrac{3}{2}}\bigr),
\end{equation}
and that the same bound holds for
$\mbox{disc}^-(G_{n,p})$.  By extending their arguments,
it is easily shown that (\ref{discrandom}) holds for
arbitrary $p\in(0,1)$. Further it is well-known that $G =
G_{n,p}$ asymptotically almost surely has $\vert \delta_p(X)\vert
= O(\sqrt{p(1 - p)n}\vert X\vert )$ for all $X \subseteq V(G_{n,p})$
(see for example~\cite{KS}), so that $G_{n,p}$ is $(p,j)$-jumbled for some $j=O(\sqrt{p(1-p)n})$.
Combining this with (\ref{discrandom}) and
Theorem~\ref{theorem: f(G) geq disc+/alpha}, we obtain that for fixed
$p\in(0,1)$ asymptotically almost surely,
\begin{equation}\label{frandom}
f(G_{n,p}) =  \Omega(n).
\end{equation}
In the other direction, results of Riordan and Selby~\cite{RS} imply that for all fixed $p\in(0,1)$, $f(G_{n,p}) \leq \beta_2 n + o(n)$ asymptotically almost surely, where $\beta_2 \approx 0.851\dots$. We believe that $f(G_{n,p})$ is concentrated around $\beta n + o(n)$ for some function $\beta=\beta_p$,
and pose the following problem.
\begin{problem}
For each fixed $p \in (0,1)$, prove the existence and determine
the value of a real number $\beta=\beta_p$ such that for
all $\delta > 0$, $\mathbb P(\vert f(G_{n,p}) - \beta_p n\vert  >
\delta  n) \rightarrow 0$ as $n \rightarrow
\infty$.
\end{problem}

\subsection{Full and co-full subgraphs}
We also a consider a variant of our problem with a
Ramsey-theoretic flavour. A subgraph $H$ of a graph $G$
is {\em co-full} if $V(H)$ induces a full subgraph of
$G^c$, the complement of $G$. Equivalently, an induced
$m$-vertex subgraph $H$ of a graph $G$ with density $p$
is co-full if it has maximum degree at most $p(m - 1)$.
Let $g(G)$ be the largest integer $m$ such that $G$ has a
full subgraph with at least $m$ vertices or a co-full
subgraph with at least $m$ vertices. In other words,
$g(G) = \max\{f(G),f(G^c)\}$. Setting $g(n) = \min\{g(G)
: \vert V(G)\vert = n\}$, we prove the following theorem:

\begin{theorem}\label{theorem: g(n)=Omega(n/log n)}
There exist constants $c_1,c_2 > 0$ such that
\[ c_1\frac{n}{\log n} \leq g(n) \leq c_2 \frac{n\log\log n}{\log n}.\]
\end{theorem}

Bounding $g(G)$ is related to, but distinct from, a
problem of Erd\H{o}s and Pach~\cite{ErdosPach83} on {\em
quasi-Ramsey numbers} (see
also~\cite{KangPachPatelRegts14}). Erd\H{o}s and
Pach~\cite{ErdosPach83} showed that for every $n$-vertex
graph $G$, in either $G$ or $G^c$ there exists a subgraph
with $m = \Omega(\frac{n}{\log n})$ vertices and minimum
degree at least $\frac{1}{2}(m - 1)$. In particular, when
$G$ has density $\frac{1}{2}$, this shows $g(G) =
\Omega(\frac{n}{\log n})$. Erd{\H os} and Pach in
addition gave an unusual weighted random graph
construction $G'$ to show their quasi-Ramsey bound was
sharp up to a $\log \log n$ factor. While $G'$ does not
have density $\frac{1}{2}$, a simple modification (see
Section \ref{g-upper}) gives a graph $G^{\star}$ of
density $\tfrac{1}{2}$ such that $g(G^{\star}) =
O(\frac{n\log\log n}{\log n})$, and this gives the upper
bound in Theorem \ref{theorem: g(n)=Omega(n/log n)}. This
leaves the following problem open:
\begin{problem}
Determine the order of magnitude of $g(n)$.
\end{problem}
By (\ref{frandom}) with $p = \frac{1}{2}$, note that
$g(G)$ is linear in $n$ for almost all $n$-vertex graphs
$G$. We may also define $g(n,p) = \min\{g(G) : \vert
V(G)\vert = n, |E(G)| = p{n \choose 2}\}$, and ask for
the order of magnitude of $g(n,p)$. Note Theorem
\ref{theorem: g(n)=Omega(n/log n)} gives $g(n,p) =
\Omega(\frac{n}{\log n})$ for all $p$.


\subsection{Relatively half-full subgraphs}

If $G$ is a graph, then a {\em relatively half-full}
subgraph of $G$ is a subgraph $H$ of $G$ such that
$d_H(v) \geq \frac{1}{2} d_G(v)$ for every $v \in V(H)$.
A key ingredient in the proof of Theorem \ref{main} is
the following theorem on relatively half-full subgraphs:

\begin{theorem} \label{relative}
Let $G$ be an $n$-vertex graph. Then $G$ contains a
relatively half-full subgraph with $\lfloor \frac{n}{2}
\rfloor$ or $\lfloor \frac{n}{2} \rfloor + 1$ vertices.
\end{theorem}

Theorem \ref{relative} is best possible, in the sense
that the smallest non-empty relatively half-full subgraph of
$K_{n}$ has $\lfloor \frac{n}{2} \rfloor + 1$ vertices
and the smallest relatively half-full subgraph of
$K_{n,n}$ has $n + 1$ vertices when $n$ is odd. For
regular graphs, we obtain:

\begin{corollary}\label{regular}
Let $G$ be an $n$-vertex $d$-regular graph. Then
$G$ contains a full subgraph with $\lfloor \frac{n}{2} \rfloor$ or $\lfloor \frac{n}{2} \rfloor + 1$ vertices.
\end{corollary}
(Note that of course $G$ itself is full, as it is regular.) When $d$ is very small relative to $n$, Alon~\cite{Alon97}
showed that any $d$-regular $n$-vertex graph contains a subgraph on $\lceil \frac{n}{2} \rceil$ vertices in which the minimum degree is at least $\frac{1}{2}d + c d^{\sfrac{1}{2}}$, exceeding the requirement for a full subgraph 
by an additive factor of $c d^{\sfrac{1}{2}}$. However, as observed by Alon~\cite{Alon97}, such a result does not hold for large $d$, as for example complete graphs and complete bipartite graphs show.

\subsection{Relatively $q$-full subgraphs}\label{section: qfull}

Let $q \in [0,1]$. A subgraph $H$ of a graph $G$ is {\em
relatively $q$-full} if $d_H(v) \geq q d_G(v)$ for all $v
\in V(H)$. We prove:
\begin{theorem}\label{q-full or (1-q)-full}
	Let $G$ be a graph on $n$ vertices. Then for every $q \in[0,1]$, $G$ contains one of the following:
	\begin{enumerate}[{\rm (i)}] \setlength\itemsep{-0.3em}
		\item a relatively $q$-full subgraph on $\lceil qn\rceil$ vertices, or
		\item a relatively $(1-q)$-full subgraph on $\lfloor (1-q)n\rfloor$ vertices, or
		\item a relatively $q$-full subgraph on $\lceil qn\rceil +1$ vertices and a relatively $(1-q)$-full subgraph on $\lfloor (1-q)n\rfloor +1$ vertices.
	\end{enumerate}
\end{theorem}
Using Theorem~\ref{q-full or (1-q)-full}, we prove
Theorem~\ref{relative} and an extension to relatively
$\frac{1}{r}$-full subgraphs for $r\geq 3$:
\begin{theorem}\label{1/r-full}
	Let $G$ be a graph on $n$ vertices, and let $r\in \mathbb{N}$. Then $G$ contains a relatively $\frac{1}{r}$-full subgraph on $\lfloor \frac{n}{r}\rfloor$, $\lceil \frac{n}{r}\rceil$ or $\lceil \frac{n}{r} \rceil+1$ vertices.
\end{theorem}
Theorem~\ref{1/r-full} is best possible in the following
sense: if $r \geq 3$, consider the complete graph $K_n$
for some $n \geq r + 2$ with $n \equiv 2 \mod r$. A
smallest non-empty relatively $\frac{1}{r}$-full subgraph of $K_n$
has exactly $\lceil \frac{n-1}{r} \rceil + 1= \lceil
\frac{n}{r}\rceil +1$ vertices.

It is natural to ask whether Theorem~\ref{1/r-full} can be extended further to cover other $q$.
\begin{problem}\label{q-full question}
Determine whether there exists a constant $c$ such that
for every $q\in [0,\frac{1}{2}]$, every $n$-vertex graph
$G$ has a relatively $q$-full subgraph with at least
$\lfloor qn\rfloor$ vertices and at most $\lfloor qn
\rfloor +c$ vertices.
\end{problem}
For $q>\frac{1}{2}$, a cycle of length $n$ shows that
there exist $n$-vertex graphs with no non-empty
relatively $q$-full subgraphs on fewer than $n$ vertices.
We might try to circumvent this example by requiring a
weaker degree condition: define a subgraph $H$ of a graph
$G$ to be {\em weakly relatively $q$-full} if $d_H(v)
\geq \lfloor q d_G(v)\rfloor$ for all $v \in V(H)$.
However even for this notion of $q$-fullness a natural
generalisation of Theorem~\ref{1/r-full} fails for
rational $q>\frac{1}{2}$: consider the second power of a
cycle of length $n$. If $x$ is a vertex in a weakly
relatively $\frac{3}{4}$-full subgraph $H$, then all but
at most one of its neighbours must also belong to $H$.
Thus vertices not in $H$ must lie at distance at least
$5$ apart in the original cycle, and $H$ must contain at
least $\frac{4}{5}n$ vertices, rather than the
$\frac{3}{4}n +O(1)$ we might have hoped for. It would be
interesting to determine whether powers of paths or
cycles provide us with the worst-case scenario for
finding weakly relatively $q$-full subgraphs when
$q>\frac{1}{2}$.
\begin{problem}\label{weakly q-full question}
	Let $q\in (\frac{1}{2},1)$. Determine whether there
exist a constant $c_q < 1$ such that every $n$-vertex
graph $G$ has a weakly relatively $q$-full $m$-vertex
subgraph where $\lfloor qn \rfloor \leq m \leq c_q n$.
\end{problem}




%

\subsection{Notation}

We use standard graph theoretic notation. In particular,
if $X,Y$ are sets of vertices of a graph $G=(V,E)$, then $e(X)$
denotes the number of edges in the subgraph $G[X]$ of $G$
induced by $X$, $e(G)$ is the number of edges in $G$, and
$e(X,Y)$ is the number of edges with one end in $X$ and the
other end in $Y$. Denote by $d_X(x)$ the number of
neighbours in $X$ of a vertex $x \in V(G)$. It is
convenient to let $\delta_p(X) = e(X) - p{\vert X\vert  \choose 2}$
when $X \subseteq V(G)$. The Erd\H{o}s-R\'{e}nyi random
graph with edge-probability $p$ on $n$ vertices is
denoted by $G_{n,p}$. If $(A_n)_{n \in \mathbb N}$ is a
sequence of events, then we say $A_n$ occurs
\emph{asymptotically almost surely} if $\lim_{n
\rightarrow \infty} \mathbb P(A_n) = 1$.

\section{Relatively $q$-full subgraphs : proofs of Theorems \ref{relative} -- \ref{1/r-full}}\label{halffull}

\begin{proof} [Proof of Theorem~\ref{q-full or (1-q)-full}]
Let $G$ be a graph on $n$ vertices, and let $q\in [0,1]$
be fixed. Let $X\sqcup Y$ be a bipartition of $V(G)$ with
$\vert X\vert= \lceil qn\rceil$ and $\vert Y\vert =
\lfloor (1-q)n \rfloor$ maximising the value of
$(1-q)e(X)+qe(Y) := M$.

If $X$ is relatively $q$-full or $Y$ is relatively
$(1-q)$-full, then we are done. Otherwise there exist $x
\in X$ and $y \in Y$ with $d_X(x)\leq \lceil
qd_G(x)\rceil-1$ and $d_Y(y) \leq \lceil (1-q) d_G(y)
\rceil-1$. Let $X'=(X\setminus\{x\})\cup\{y\}$ and
$Y'=(Y\setminus\{y\})\cup\{x\}$, and let $\mathds{1}_{xy} =
1$ if $\{x,y\} \in E(G)$ and $\mathds{1}_{xy} = 0$
otherwise.  Write $d_G(x) = d(x)$ for $x \in V(G)$ and set $M'=
(1 - q)e(X') + qe(Y')$. Then
\begin{eqnarray*}
M' &=& M + (1-q)d_X(y)-qd_Y(y) +qd_Y(x)-(1-q)d_X(x)-\mathds{1}_{xy}\\
&=& M+ (1 - q)d(y) - d_Y(y) + qd(x) - d_X(x) -\mathds{1}_{xy} \\
&\geq& M + \Bigl((1-q)d(y)-\lceil(1-q)d(y)\rceil\Bigr) + \Bigl(qd(x) -\lceil qd(x)\rceil\Bigr) + 2 - \mathds{1}_{xy}.
\end{eqnarray*}
Since $X\sqcup Y$ maximised $(1-q)e(X)+qe(Y)$ over all
bipartitions with $\vert X\vert = \lceil qn \rceil$,
$\vert Y\vert=\lfloor (1-q)n\rfloor$, $M' \leq M$ and we
deduce from the inequality above that
\begin{center}
\begin{tabular}{lp{5in}}
(a) & $d_X(x)=\lceil qd(x)\rceil-1$ and $d_Y(y) = \lceil (1-q)  d(y) \rceil-1$. \\
(b) & $\{x,y\} \in E(G)$.  \\
(c) & $(1-q)d(y)<\lceil (1-q)d(y)\rceil$ and $qd(x)< \lceil qd(x)\rceil$.
\end{tabular}
\end{center}
 Now let $B_X$ denote the set of $x\in X$ with
$d_X(x)\leq \lceil qd(x)\rceil-1$ and $B_Y$ the set of
$y\in Y$ with $d_Y(y)\leq \lceil (1-q)d(y)\rceil-1$. By
our assumption, both sets are non-empty. By (b) above, $B_X\sqcup B_Y$ induces a complete
bipartite subgraph of $G$. Thus by (a) we have that for every $x\in B_X, y\in
B_Y$, $X\cup\{y\}$ is a relatively $q$-full subgraph on
$\lceil qn\rceil+1$ vertices and $Y\cup \{x\}$ is a
relatively $(1-q)$-full subgraph on $\lfloor
(1-q)n\rfloor+1$ vertices.
\end{proof}

\begin{proof}[Proof of Theorem~\ref{relative}]
Apply Theorem~\ref{q-full or (1-q)-full} with $q=\frac{1}{2}$.
\end{proof}

\begin{proof}[Proof of Corollary \ref{regular}] Suppose at least one of $n,d$ is odd. By Theorem \ref{relative}, every $d$-regular graph has an $m$-vertex subgraph $H$ with $m \in \{\lfloor \frac{1}{2}n \rfloor,\lfloor \frac{1}{2}n \rfloor + 1\}$
such that $d_H(v) \geq \lceil \frac{d}{2} \rceil$ for every $v \in V(H)$. Since for $n,d$ not both even
\[ \Big\lceil \frac{d}{2} \Big\rceil = \Big\lceil \frac{d}{n - 1} \Big\lfloor \frac{n}{2} \Big\rfloor \Big\rceil \geq \Bigl\lceil \frac{d}{n - 1}(m - 1)\Big\rceil , \]
the subgraph $H$ is a full subgraph.

In the case where both $n$ and $d$ are even, we need to
use a slightly stronger form of Theorem~\ref{relative}.
In the particular case where $G$ is $d$-regular with $d$
even and $q=\frac{1}{2}$, condition (c) in the proof of
Theorem~\ref{q-full or (1-q)-full} cannot be satisfied,
and in particular one of the alternatives (i) or (ii)
must hold in Theorem~\ref{q-full or (1-q)-full}. Thus $G$
must contain a subgraph $H$ on $\frac{n}{2}$ vertices
with minimum degree at least $\frac{d}{2}$, which is a
full subgraph.
\end{proof}

\begin{proof}[Proof of Theorem~\ref{1/r-full}]
We use Theorem~\ref{q-full or (1-q)-full} and induction on $r$. The base case $r=1$ is trivial, and Theorem~\ref{relative} deals with the case $r=2$. Now apply Theorem~\ref{q-full or (1-q)-full} with $q=\frac{1}{r}$: given a graph $G$ on $n$ vertices, this gives us a $\frac{1}{r}$-full subgraph on $\lceil \frac{n}{r}\rceil$ or $\lceil \frac{n}{r}\rceil +1$ vertices (alternatives (i) and (iii)) or an $\frac{r-1}{r}$-full subgraph $H$ on $\lfloor \frac{r-1}{r}n \rfloor$ vertices (alternative (ii)). In the latter case, we use our inductive hypothesis to find a $\frac{1}{r-1}$-full subgraph $H'$ of $H$ on $m$ vertices, for some $m:\ \lfloor \frac{n}{r}\rfloor\leq m\leq \lceil \frac{n}{r}\rceil +1$. The subgraph $H'$ is easily seen to be a $\frac{1}{r}$-full subgraph of $G$, and so we are done.
\end{proof}

\section{A greedy algorithm : proof of Theorem \ref{greedy}}\label{proofofgreedy}

A natural strategy for obtaining a full subgraph in a
graph $G$ of density $p$ on $n$ vertices is to repeatedly
remove vertices of relatively low degree. When there are
$i$ vertices left in the graph, such a greedy algorithm
finds a vertex of degree at most $\lceil p(i - 1)\rceil -
1$ and deletes that vertex, unless no such vertex exists,
in which case the $i$ vertices induce a full subgraph. If
$G$ has positive discrepancy $\alpha$, then we apply this
algorithm in a subgraph $H$ on $m$ vertices with $e(H)
\geq p{m \choose 2} + \alpha$ to obtain Theorem
\ref{greedy}.

\begin{proof}[Proof of Theorem \ref{greedy}] If $G$ has positive discrepancy $\alpha>0$, then its density $p$ is strictly less than $1$. Let $H$ be a
 subgraph of $G$ with $m$ vertices such that $e(H)= p{m \choose 2} + \alpha$. At stage $i$ we delete a vertex of degree at most $\lceil p(m - i)\rceil - 1$ in the remaining graph, or stop if no such vertex exists. The number of edges remaining after stage $i$ is at least
\[ p{m \choose 2} + \alpha - \sum_{j = 1}^i p(m - j) = p{m \choose 2} + \alpha - p{m \choose 2} + p{m - i \choose 2} =\alpha +p{m - i \choose 2}.\]
Therefore the greedy algorithm must terminate with a full
subgraph on $m - i$ vertices for some $i$ satisfying
$(1-p){m - i \choose 2} \geq \alpha$. We conclude $f(G)
\geq m - i \geq  (1 - p)^{-\sfrac{1}{2}}
(2\alpha)^{\sfrac{1}{2}}$.
\end{proof} An alternate proof may be obtained by
appealing to Lemma~\ref{obs: max-disc implies
full/co-full}, which states that a subgraph attaining the
maximum positive discrepancy must be full. The example of
a clique with $m$ vertices and $n - m$ isolated vertices
which shows that Theorem~\ref{greedy} is tight is the
same example which shows $f(n,p) = O(p^{\sfrac{1}{2}}n)$
for $p \leq n^{-\sfrac{2}{3}}$. We now prove that
$|f(n,p) - p^{\sfrac{1}{2}}n| \leq 1$ for $p \leq
n^{-\sfrac{2}{3}}$.

\begin{proof}[Proof that $|f(n,p) - p^{\sfrac{1}{2}}n| \leq 1$ for $p \leq n^{-\sfrac{2}{3}}$]
First we show $f(n,p) < p^{\sfrac{1}{2}}n + 1$ for all $p
: 0 < p \leq 1$. If $m$ is defined by ${m - 1 \choose 2}
< p{n \choose 2} \leq {m \choose 2}$, then the $n$-vertex
graph $G$ consisting of a subgraph of a clique of size
$m$ with $p{n \choose 2}$ edges, together with $n - m$
isolated vertices has
   $f(G) \leq m \leq p^{\sfrac{1}{2}}n + 1$. Next we show that every $n$-vertex graph $G$ of density $p$ has a full subgraph with at least $p^{\sfrac{1}{2}}n - 1$ vertices if $p \leq n^{-\sfrac{2}{3}}$. Remove all isolated vertices from $G$. The number of
isolated vertices is clearly at most $n - p^{\sfrac{1}{2}}n$, otherwise the remaining graph has
$p{n \choose 2}$ edges and fewer than $p^{\sfrac{1}{2}}n$ vertices, which is impossible since this is denser than a complete graph.
So we have a subgraph $H$ with at least $p^{\sfrac{1}{2}}n$ vertices and $p{n \choose 2}$ edges with no isolated vertices.
Clearly $H$ has a subgraph of minimum degree at least $1$ with at least $p^{\sfrac{1}{2}}n - 1$ vertices and at most
$p^{\sfrac{1}{2}}n + 1$ vertices, since the removal of a leaf in a spanning forest creates at most one new isolated vertex.
This subgraph is full since
\[ \lceil p (p^{\sfrac{1}{2}}n+1 -1) \rceil = \lceil p^{\sfrac{3}{2}}n \rceil \leq 1\]
when $p \leq n^{-\sfrac{2}{3}}$, as required.
\end{proof}

{\bf Remarks.} The analysis of the greedy algorithm in the proof of Theorem~\ref{greedy} above is not optimal; in fact by considering the asymptotic behavior
of
\[ \phi = \liminf_{n \rightarrow \infty} \frac{1}{n} \sum_{i = 1}^n (p(n - i) + 1 - \lceil p(n-i) \rceil),\]
and performing our greedy algorithm directly on $G$ rather than on a maximum discrepancy subgraph it follows that
\[f(n,p) \geq\left\{ \begin{array}{ll}
\Bigl(\frac{n(q + 1)}{q(1 - p)}\Bigr)^{\sfrac{1}{2}} & \textrm{if $p$ is rational with denominator $q > 1$  }\\
\Bigl(\frac{n}{1 - p}\Bigr)^{\sfrac{1}{2}}  & \textrm{if
	$p<1-\varepsilon$ for some $\varepsilon>0$ and $p
\geq \frac{1}{n}$.}\end{array} \right.\] For instance if
say $p_n = \frac{1}{2} - o(1) < \frac{1}{2}$, then
$f_{p_n}(n) \geq (1 - o(1))\sqrt{2n}$, as shown by
Erd\H{o}s, \L uczak and Spencer~\cite{ELS}. These lower
bounds on $f(n,p)$ will be superseded by the better
bounds given in Theorem \ref{main}.

We note that there exist examples of $n$-vertex graphs
with density $p=\frac{1}{2}+o(1)$ where greedily removing
a vertex of minimal degree could yield a full subgraph of
order only $O(\sqrt{n})$. Consider the graph $G$ on
$V=\{0,1,....,4n+1\}$ obtained by taking the
$n^{\textrm{th}}$ power of the Hamiltonian cycle through
$0,1,2, \ldots , 4n+1$, adding edges between all
antipodal pairs $\{i, i+(2n+1)\}$ (with addition modulo
$4n+2$), and adding a complete bipartite graph $K_{m,m}$
with parts $\{0,1,...,m-1\}$ and
$\{2n+1,2n+2,...,2n+m\}$, where $m =
(3n)^{\sfrac{1}{2}}+O(1)$. It is an easy exercise to show
that by removing antipodal pairs of minimum degree
vertices a greedy algorithm could fail to find a full
subgraph until it has stripped the graph down to the
planted complete bipartite graph $K_{m,m}$.


\section{Jumbledness: proof of Theorem~\ref{theorem: f(G) geq disc+/alpha}}\label{section: jumbledness}
As our arguments involve passing to subgraphs with different edge densities, it shall be useful to adapt our notion of full subgraphs, discrepancy and jumbledness as follows.

Let $G=(V,E)$ be a graph. An induced subgraph $H$ of $G$
on $m$ vertices is called \emph{$p$-full} if its minimum
degree is at least $p(m-1)$, and is called
\emph{$p$-co-full} if its maximum degree is at most
$p(m-1)$. Let $f_p(G)$ be the largest number of vertices
in a $p$-full subgraph of $G$, and let $g_p(G)$ be the
largest size of a $p$-full or $p$-co-full subgraph of
$G$. If $p$ happens to be the density of $G$, then in
fact $f_p(G)= f(G)$ and $g_p(G) = g(G)$. Determining the
smallest possible values of $f_p(G)$ and $g_p(G)$ given
the number of edges and number of vertices in $G$ can be
viewed as generalisations of the Tur\'{a}n and of the
Ramsey problems, respectively, which comprise the case $p
= 1$; see~\cite{KangPachPatelRegts14} and the references
therein.

The \emph{positive $p$-discrepancy} of $G$ is defined to
be $\mathrm{disc}_p^+(G)=\max_{X\subseteq
V}\delta_p(X)$. The \emph{negative $p$-discrepancy} of
$G$ is $\mathrm{disc}_p^-(G)=\max_{X\subseteq V}
(-\delta_p(X))$. The \emph{$p$-discrepancy} of $G$ is
$\mathrm{disc}_p(G)=\max\left(\mathrm{disc}_p^+(G),
\mathrm{disc}_p^-(G)\right)$. Finally, the
\emph{$p$-jumbledness} of $G$ is
\[j_p(G)=\max_{X\subseteq V} \frac{\vert \delta_p(X)\vert}{\vert X\vert}.\]
We begin with the following simple observation:
\begin{lemma}\label{obs: max-disc implies full/co-full}
Let $X$ be a subset of $V(G)$ such that
$\delta_p(X)=\mathrm{disc}_p^+(G)$.
Then $G[X]$ is $p$-full.
\end{lemma}

\begin{proof}
Indeed, otherwise deleting a minimum degree vertex
from $X$ would strictly increase the $p$-discrepancy.
\end{proof}

We shall prove the following, slightly more general form
of Theorem~\ref{theorem: f(G) geq disc+/alpha}.
\begin{theorem}\label{theorem: p-full subgraph order at least disc/jumb}
Let $G$ be a graph and let $p \in[0,1]$. Then
\[f_p(G)\geq \frac{\mathrm{disc}^+_p(G)}{j_p(G)} \quad
\mbox{ and } \quad g_p(G)\geq \frac{\mathrm{disc}_p(G)}{j_p(G)}.\]
\end{theorem}
\begin{proof} 
Let $X$ be a subset of $V(G)$ such that
$\delta_p(X)=\mathrm{disc}_p^+(G)$.
Then $G[X]$ is $p$-full by Lemma \ref{obs: max-disc
implies full/co-full}, and in particular, $f_p(G) \geq
|X|$. By definition of $p$-jumbledness we have
\[\mathrm{disc}_p^+(G) = \Bigl\vert \delta_p(X) \Bigr\vert \leq j_p(G) \vert X\vert \leq j_p(G) f_p(G).\]
Applying the resulting lower bound for $f_p(G)$ to the
complement $G^c$ of $G$ (and noting that
$\mathrm{disc}_p^-(G)=\mathrm{disc}_{1-p}^+(G^c)$ and
$j_{p}(G)=j_{1-p}(G^c)$), we have
\[g_p(G) \geq \max\left( \frac{\mathrm{disc}_p^+(G)}{j_p(G)}, \frac{\mathrm{disc}_{1-p}^+(G^c)}{j_{1-p}(G^c)}\right)=
\frac{\mathrm{disc}_p(G)}{j_p(G)}.\]
\end{proof}

\section{Proof of Theorem~\ref{theorem: g(n)=Omega(n/log n)}}
Let $G=(V,E)$ be a graph on $n$ vertices, and let $k:\
1\leq k\leq n$ be an integer. The \emph{$p$-jumbledness
of $G$ on $k$-sets} is defined to be
\[j_{k,p}(G)=\max_{X\subseteq V: \ \vert X\vert =k} \frac{\vert \delta_p(X) \vert}{\vert X\vert}.\]
Similarly, the \emph{positive $p$-discrepancy of $G$ on $k$-sets} is defined to be
\[\mathrm{disc}_{k,p}^+(G)=\max_{X\subseteq V: \ \vert X\vert =k} \delta_p(X),\]
with the \emph{negative $p$-discrepancy on $k$-sets} $\mathrm{disc}_{k,p}^-(G)$ and the \emph{$p$-discrepancy on $k$-sets} $\mathrm{disc}_{k,p}(G)$
defined mutatis mutandis.
Note that by definition we have $j_p(G)=\max \{j_{k,p}(G): \ 1\leq k\leq n\}$ and $\mathrm{disc}_p(G)=\max\{\mathrm{disc}_{k,p}(G): \ 1\leq k\leq n\}$. In addition for each $k$: $1\leq k\leq n$ we have $\mathrm{disc}_{k,p}(G) = k\cdot  j_{k,p}(G)$.

In general, it is not true that the jumbledness on
$k$-sets is of the same order as the global jumbledness
of $G$. Indeed consider an Erd{\H o}s--R\'enyi random
graph $G$ with edge probability $p=\frac{1}{2}$ within
which we plant a clique and a disjoint independent set,
each of order $m=n^{\frac{3}{4}}$. It is straightforward
to show with $k = \lceil \frac{n}{2}\rceil$ that asymptotically almost
surely
$j_{k,\sfrac{1}{2}}(G) = \Theta(n^{\sfrac{1}{2}})$ whilst
$j_{\sfrac{1}{2}}(G) = \Theta(m)$. To prove Theorem~\ref{theorem: g(n)=Omega(n/log
n)}, we use the following theorem of
Thomason~\cite{Thomason87} to show that every $n$-vertex
graph of density $p$ has a subgraph $G^{\star}$ with $n^{\star} = \Omega(n)$
vertices such that the $p$-jumbledness of $G^{\star}$ on $\lceil
\frac{n^{\star}}{2} \rceil$-sets and the global $p$-jumbledness of $G^{\star}$
differ by a factor of only $O(\log n)$.

\begin{proposition}[Thomason~\cite{Thomason87}]\label{theorem:
Thomason} Let $G$ be a graph of order $n$, let $\eta n$
be an integer between $2$ and $n-2$, and let $M>1$ and $p\in[0,1]$. Suppose that every set $X$ of $\eta n$
vertices of $G$ satisfies
\[\left\vert \delta_p(X) \right\vert \leq \eta n \alpha.\]
Then $G$ contains an induced subgraph $G'$ of order
\[\bigl\vert V(G')\bigr\vert \geq \left(1-\frac{880}{\eta (1-\eta)^2 M}\right)n\]
such that $G'$ is $(p, M\alpha)$-jumbled.
\end{proposition}

\subsection{Proof of $g(n) = \Omega(\frac{n}{\log
n})$.}\label{g-lower}

Let $G = G_0$ be a graph on $n = n_0$ vertices with
density $p$. Let $\eta = \frac{1}{e}$ and $M=880\eta(1 -
\eta)^2(\log n + 1)$. For $i\geq 0$, if
\begin{align}\label{eq: jumbledness not achieved on large sets}
j_{\lceil {\sfrac{n_i}{2}}\rceil,p}(G_i) < \lfloor \tfrac{\eta}{M} j_p(G_i)\rfloor,
\end{align}
then apply Proposition~\ref{theorem: Thomason} to find an
induced subgraph $G_{i+1}$ of $G_i$ on $n_{i+1} \geq
(1-\frac{880}{\eta(1 - \eta)^2 M})n_i$ vertices  with
\begin{align}\label{eq: jumbledness in induced subgraph}
j_p(G_{i+1}) \leq M j_{\lceil {\sfrac{n_i}{2}}\rceil,p}(G_i).
\end{align}
Combining (\ref{eq: jumbledness not achieved on large sets}) and (\ref{eq: jumbledness in induced subgraph}) and iterating, we see that
\begin{align}\label{eq: decay of jumbledness}
j_p(G_{i+1}) \leq M j_{\lceil {\sfrac{n_i}{2}}\rceil,p}(G_i) < \lfloor \eta j_p(G_i)\rfloor \leq \lfloor \eta^{i+1} j_p(G_0) \rfloor \leq
\lfloor e^{-i-1}(n/2) \rfloor.
\end{align}
Since $j_p(G_{i + 1}) \geq 0$, we deduce from (\ref{eq:
decay of jumbledness}) that this procedure must terminate
for some $i <\log n$ with a graph $G^{\star} = G_i$
on $n_i$ vertices where
\begin{eqnarray*}
n_i &\geq \left(1-\frac{880}{\eta(1 - \eta)^2 M}\right)^{\log
n}n =\left(1 - \frac{1}{\log n + 1}\right)^{\log n} n  \geq  \frac{n}{e}
\end{eqnarray*}
and where $j_{\lceil
{\sfrac{n_i}{2}}\rceil,p}(G^{\star})\geq \frac{\eta}{M}
j(G^{\star})$. Applying Theorem~\ref{theorem: p-full
subgraph order at least disc/jumb} to $G^{\star}$ we have
\[g(G)\geq g_p(G^{\star})\geq \frac{\mathrm{disc}_p(G^{\star})}{j_p(G^{\star})}\geq \frac{\mathrm{disc}_{\lceil {\sfrac{n_i}{2}}\rceil,p}(G_i)}{eM j_{\lceil {\sfrac{n_i}{2}}\rceil,p}(G_i)+1}
\geq \frac{n_i}{4e M} \geq \frac{n}{4e^2 M}.\] Since $M =
880\eta(1 - \eta)^2 (\log n + 1) = \Theta(\log n)$, this
completes the proof. \qed

\subsection{Proof of $g(n) = O\left(\frac{n\log\log n}{\log
n}\right)$.}\label{g-upper}

We shall base our argument on a result of Erd{\H o}s and Pach~\cite{ErdosPach83}. Using an unusual weighted random graph construction, they proved the following theorem: 
\begin{proposition}[Erd{\H o}s--Pach~\cite{ErdosPach83}]\label{construction: Erdos-Pach}
For any $v>0$ there exists $c>0$ such that for all $n$ sufficiently large there exists an $n$-vertex graph $G$ in which every $m$-set of vertices with $m\geq cn\log\log n/\log n$ induces a subgraph with minimum degree strictly less than ${\sfrac{1}{2}}(m-1)- m^{1-v}$ and maximum degree strictly greater than ${\sfrac{1}{2}}(m-1)+m^{1-v}$.
\end{proposition}
We note that the formulation of Erd{\H o}s and Pach's result given above is due to Kang, Pach, Patel and Regts, see Theorem~4 in~\cite{KangPachPatelRegts14} and Section~3 in the same paper for an exposition of the delicate calculations involved. Erd{\H o}s and Pach's construction almost gives us what we want, namely a graph with no large full or co-full subgraph, with one caveat: it does not have density $\frac{1}{2}$. We circumvent this problem by taking disjoint copies of their construction and its complement and adding a carefully chosen random bipartite subgraph with density $\frac{1}{2}$ between them.

Fix $v=\frac{1}{7}$, and let $2n\in \mathbb{N}$ be sufficiently large to ensure the existence of a graph $G_{2n}$ with the properties guaranteed by Proposition~\ref{construction: Erdos-Pach}. Let $A$ and $B$ be disjoint sets of $2n$ vertices, each split into $n$ pairs. Given a pair in $A$ and a pair in $B$, place one of the two possible matchings between them selected uniformly at random, and do this independently for each of the $n^2$ pairs (pair from $A$, pair from $B$). This gives a random bipartite graph $H$ between $A$ and $B$ with density precisely $\sfrac{1}{2}$. Add a copy of $G_{2n}$ to $A$ and a copy of its complement $G_{2n}^c$ to $B$ to obtain a graph $G^{\star}$ on $4n$ vertices with density exactly $\sfrac{1}{2}$.

Let $m= \lceil cn\log \log n/\log n\rceil$,  and let $2\lambda=m^{1-v}=m^{\sfrac{6}{7}}$. Let $Y$ be a set of at least $4m$ vertices in $G^{\star}$. Without loss of generality, assume that $l=\vert Y\cap B\vert \leq \vert Y\cap A\vert$ and thus $\vert Y\cap A\vert \geq 2m$. Set $X'$ to be the collection of vertices in $Y\cap A$ that have degree at most ${\frac{1}{2}}\left(\vert Y\cap A\vert-1\right) -\lambda$ in $Y\cap A$. There are at least $\lambda$ such vertices, for otherwise $X= (Y\cap A)\setminus X'$ is a set of at least $2m-\lambda>m$ vertices inducing a subgraph of $G_{2n}$ with minimum degree at least ${\frac{1}{2}}\left(\vert Y\cap A\vert-1\right)-2\lambda >{\frac{1}{2}}\left(\vert X\vert-1\right) -\vert X\vert^{\sfrac{6}{7}}$,  a contradiction. For each pair from $A$ discard if necessary one of its two vertices from $X'$ to obtain a set $X''\subseteq X'$ of at least $\lambda/2$ vertices, each coming from a distinct pair. Note that by construction this means the degrees of the vertices from $X''$ into $Y\cap B$ are  independent random variables with mean $\frac{1}{2}l$. 

For $Y$ to induce a full subgraph of $G$, each vertex in $X'$ would need to have at least 	$\frac{1}{2}l +\lambda$ neighbours in $Y\cap B$. By standard concentration inequalities (e.g. the Chernoff bound), the probability that a given vertex in $X''$ has that many neighbours in $Y\cap B$ is at most $\exp\left(-\lambda^2/2l\right)$. By the independence noted above, the probability that all vertices in $X''$ have the right degree in $Y\cap B$ is thus at most
\[\exp\left(-\frac{\lambda^2}{2l}\vert X''\vert \right)\leq \exp\left(-\frac{\lambda^3}{4l} \right)\geq \exp\left(-\frac{1}{8}n^{\sfrac{11}{7}}(\log n)^{-3} \right)=o(2^{-4n}).\] 
It follows that asymptotically almost surely there is no pair $(X', Z)$ where $X'\subseteq A$ is a collection of at least $\lambda$ vertices and $Z\subseteq B$ are such that every vertex in $X'$ has degree at least $\frac{1}{2}\vert Z\vert+\lambda $ in $Z$. By symmetry, asymptotically almost surely no such pair exists either when $X'\subseteq B$ and $Z\subseteq A$, and in particular $G^{\star}$ contains no full subgraph on $4m$ vertices.	Still by symmetry, asymptotically almost surely the complement ${\left(G^{\star}\right)}^c$ also fails to contain a full subgraph on $4m$ vertices, and we deduce that $g(G^{\star})<4m$
as desired.

The construction of $G^{\star}$ shows $g(4n)=O\left(4n \log \log (4n)/\log (4n)\right)$, and it is straightforward to adapt it to show that more generally $g(n)=O\left(n \log \log n/\log n\right)$ when $n\not\cong 0 \mod 4$.
\qed

\section{Proof of Theorem \ref{main}}\label{proofofmain}

\begin{proof}[Proof of $f(n,p) = O(n^{\sfrac{2}{3}})$ for $p = \frac{r}{r + 1} + cn^{-\sfrac{2}{3}}$ and $c \geq 1$ fixed]
Let $n\in \mathbb{N}$, and let $p=\frac{r}{r + 1} + cn^{-\sfrac{2}{3}}$ for some $c\geq 1$ be such that $p\binom{(r+1)n}{2}\in \mathbb{N}$. Set $\delta = cn^{-\sfrac{2}{3}}$.
Take a complete $(r + 1)$-partite graph with parts $S_1,S_2,\dots,S_{r + 1}$, and $n$ vertices in each part,
and for $i = 1,2,\dots,r + 1$, add a clique $T_i$ of size $k$ in $S_i$, such that
\[ {r + 1 \choose 2} n^2 + (r + 1){k-1 \choose 2}< p{(r + 1)n \choose 2} \leq {r + 1 \choose 2} n^2 + (r + 1){k \choose 2}.\]
A quick calculation shows $k\geq \sqrt{\delta r}n$ for $n$ sufficiently large. Delete edges from the $T_i$ in an equitable manner as necessary to obtain a graph $G_n$ on $(r+1)n$ vertices with precisely $p\binom{(r+1)n}{2}$ edges.
Suppose $H$ is a full subgraph of $G_n$ on $m>(r + 1)k$ vertices, induced by sets
$X_i \subseteq S_i$ where $X_i$ has size $s_i$ for $i = 1,2,\dots,r + 1$. Without loss of generality, we may assume $s_1=\max_i s_i >k$. For a vertex $v \in X_1$, let $d_j(v)$ denote the number of neighbours of $v$ in $X_j$.
Then for $x \in X_1 \setminus V(T_1)$, which is non-empty since $s_1 > k$,
\[ d_H(x) = \sum_{j = 2}^{r + 1} d_j(x) \leq m-s_1\leq \frac{r}{r+1}m.\]
On the other hand, since $H$ is full,
\[ d_H(x) \geq p(m - 1) = \frac{r}{r+1}m + \delta m - p.\]
It follows that $\delta m \leq p$ and thus $m\leq \delta^{-1}p\leq c^{-1}n^{\sfrac{2}{3}}$. However we had assumed that  
$m>(r+1)k> r^{\sfrac{3}{2}}\sqrt{c} n^{\sfrac{2}{3}}$. Taken together, our bounds for $m$ imply $c^{\sfrac{3}{2}}<r^{-\sfrac{3}{2}}$, and in particular $c<1$, a contradiction.
Thus
\[ f(G_n) \leq (r + 1)k = O\left(((r+1)n)^{\sfrac{2}{3}}\right).\]
This proves the second part of Theorem \ref{main}. \end{proof}

\begin{proof}[Proof of $f(n,p) \geq \frac{1}{4}(1 - p)^{\sfrac{2}{3}}n^{\sfrac{2}{3}}-1$ for $p = p_n: \   n^{-\sfrac{2}{3}}< p_n < 1- n^{-\sfrac{1}{7}}$] Let $G$ be an $n$-vertex graph of density $p$. We shall repeatedly delete vertices of minimum degree to obtain a sequence of subgraphs $G=G_1, G_2, G_3, \ldots$, with $G_i$ having $n-i+1$ vertices.

Let $m = \lceil \frac{n}{2} \rceil$ and $d_i = \lceil p(n - i) \rceil$. Note that $d_i$ is the minimum degree required for $G_i$ to be full.
Let $t$ be a positive integer so that $(1-p)^{-\sfrac{2}{3}}n^{\sfrac{1}{3}} \leq 2^t < 2(1-p)^{-\sfrac{2}{3}}n^{\sfrac{1}{3}}$, and let $r_i$ be the remainder when $d_i$ is divided by $2^t$. 
For at least $\frac{(1-p)}{2}m$ of the values $i: \ 1\leq i\leq m$, we have $r_i \leq (1-p)2^t$.
At stage $i \leq m$ of the algorithm, we delete a vertex of minimum degree from $G_i$.
If for some $i \leq m$ such that $r_i \leq (1-p) 2^t$, all $n-i+1$ vertices in the graph $G_i$ have degree at least $d_i - r_i + 1$, then,
by Theorem~\ref{1/r-full} (or Theorem~\ref{relative} applied $t$ times), $G_i$ has a $\frac{1}{2^t}$-full subgraph $H$ on $N$ vertices, where
\[\Bigl \lfloor \frac{n-i+1}{2^t} \Bigr\rfloor \leq N \leq \Bigl\lceil \frac{n-i+1}{2^t}\Bigr \rceil+1 \leq \frac{n-i}{2^t} +2 -\frac{1}{2^t}.\]
Write $d_i =q2^t+r_i$. The minimum degree in $H$ is
\begin{align}
D&\geq \Bigl \lceil \frac{d_i-r_i+1}{2^t} \Bigr \rceil= q+1.\label{inequality: mindeg in relatively 2^{-t}-full subgraph}
\end{align}
For $H$ to be a full subgraph of $G$ we require $D\geq p(N-1)$. Now
\begin{align}
p(N-1)& \leq p\left(\frac{n-i}{2^t} +1 -\frac{1}{2^t}\right)\notag\\
&< \frac{d_i}{2^t}+p=q+\frac{r_i}{2^{t}}+p, \label{inequality: requirement for full}
\end{align}
which is at most $q+1$ since $r_i \leq (1-p)2^t$. As this is at most our lower bound on $D$, $H$ is a full subgraph of $G$. Our choice of $t$ ensures
\[ \vert V(H)\vert  \geq \Bigl\lfloor\frac{m}{2^t}\Bigr\rfloor\geq \frac{(1-p)^{\sfrac{2}{3}} n^{\sfrac{2}{3}}}{4}-1.\]


On the other hand suppose that at every stage $i \leq m$
of the greedy algorithm where $r_i \leq (1-p) 2^{t}$, we
could remove a vertex of degree at most $\lceil p(n - i)
\rceil - r_i$ and that at every other stage $i\leq m$ we
could remove a vertex of degree at most $\lceil
p(n-i)\rceil -1$ (for otherwise we would have found a
full subgraph on at least $m$ vertices). Set $I=\{i\leq
m: \ r_i \leq (1-p) 2^{t}\}$. We know that $\vert I \vert
\geq \frac{(1-p)m}{2}$. What is more, $I$ can be divided
into intervals of consecutive indices $i$ of length at
most $(1-p)2^t\cdot(\frac{1}{p})$, and over each of these
intervals $r_i$ takes each of the values $1,2, \ldots
\lfloor(1-p)2^t\rfloor$ at least $\frac{1-p}{p}$ times.
Indeed, suppose $r_{i-1}=j+1$ and $r_i=j$ for some $j\geq
1$. Then there is a  $k$: $\frac{1-p}{p}\leq k \leq
\frac{1}{p}$ such that $r_{i'}=j$ for $i'\in \{i, i+1,
\ldots i+k-1\}$ and $r_{i+k}=j-1$.

By considering $\sum r_i$ on these intervals and using  $m=\lceil\frac{n}{2}\rceil$, $(1-p)^ {-\sfrac{2}{3}}n^{\frac{1}{3}}\leq 2^ t<2(1-p)^ {-\sfrac{2}{3}}n^{\frac{1}{3}}$, we get that:
\begin{eqnarray*}
\alpha := \mbox{disc}^+(G) &\geq& \left(p{n \choose 2} - \sum_{i =1}^{m} (\lceil p(n - i)\rceil - 1) + \sum_{i\in I} (r_i-1)\right) -p\binom{m}{2}\\
&\geq & \Bigl \lfloor \left(\frac{(1-p)m}{2} \right)\Big/ \left(\frac{(1-p)2^t}{p}\right) \Bigr\rfloor \cdot   \left(\sum_{j = 1}^{\lfloor (1-p) 2^{t}\rfloor } \frac{1-p}{p}j\right) \\
&\geq & \Bigl \lfloor
\frac{p(1-p)^{\sfrac{2}{3}}n^{\sfrac{2}{3}}}{8}
\Bigr\rfloor  \left(\frac{1-p}{2p}\right)\left(\lfloor(1-p)^{\sfrac{1}{3}}n^{\sfrac{1}{3}} \rfloor \right)\left(\lfloor(1-p)^{\sfrac{1}{3}}n^{\sfrac{1}{3}} \rfloor  +1\right)\\
 &\geq &\frac{(1-p)^{\sfrac{7}{3}}}{32} n^{\sfrac{4}{3}}.
 \end{eqnarray*}
Then, by Theorem \ref{greedy}, we have
\[ f(G) \geq \sqrt{\frac{2\alpha}{1-p}} \rfloor  \geq \frac{(1-p)^{\sfrac{2}{3}}}{4} n^{\sfrac{2}{3}}.\]
This completes the proof of Theorem \ref{main}. \end{proof}

{\bf Remark.} We did not optimise the constants in the
proof of $f(n,p) =
\Omega({(1-p)}^{\sfrac{2}{3}}n^{\sfrac{2}{3}})$, since it
is unlikely that this argument gives an asymptotically
tight lower bound on $f(n,p)$. Also note that for
$p=1-o(n^{-\sfrac{1}{7}})$, this lower bound on $f(n,p)$
is superseded by that given in Theorem~\ref{greedy}.

\section{Concluding remarks}

$\bullet$ We showed that $f(n,p) = \Omega(n^{\sfrac{2}{3}})$
and that this bound is tight up to constants for many values of
$p$. It remains an open problem to determine the order of
magnitude of $f(n,p)$ for each $p = p(n)$. Similarly, we
leave it as an open problem to determine the order of
magnitude of $g(n)$ and $g(n,p)$, having proved $g(n) = \Omega(n/\log
n)$ and $g(n)=O(n\log \log n/\log n)$.

$\bullet$ One motivation for studying relatively
half-full subgraphs, apart from their use in the proof of
Theorem \ref{main}, is a random process on graphs known
as majority bootstrap percolation: vertices of a graph
are infected at time zero with probability $p$, and at
any later time a vertex becomes infected if more than
half of its neighbours are infected. A key quantity of
interest in bootstrap percolation is the function
$\theta_p(G)$, which is the probability that the process
on the graph $G$ infects all the vertices in finite time.
The quantity $\theta_p(G)$ is precisely the probability
that at time zero there is no relatively half-full
subgraph of uninfected vertices. One may ask whether
there exist $p
> 0$ and $c < 1$ such that for every graph $G$, $\theta_p(G) \leq c$. In other words, is it the case that
if the infection probability is too small (but still
positive), then there is an absolute positive probability
that on any graph $G$ we fail to infect any vertex from
some relatively half-full subgraph (and hence that the
infection does not spread to all vertices of $G$)?
Recently, this was answered in the negative by Mitsche,
P{\'e}rez-Gim\'{e}nez and Pra{\l}at~\cite{MPP}, who given
any arbitrarily small $p>0$ constructed a sequence of
regular graphs $G_1,G_2,\dots $ such that $\lim_{n
\rightarrow \infty} \theta_p(G_n) = 1$. Another question, which remains open, is
whether the number of relatively half-full subgraphs of
an $n$-vertex graph grows at rate $\exp(\Theta(n))$ as $n
\rightarrow \infty$.


$\bullet$ {\bf Hypergraphs.} If $G$ is an $n$-vertex
$r$-uniform hypergraph of density $p$, then an $m$-vertex
subgraph $H$ of $G$ is {\em full} if for every $v \in
V(H)$, $d_H(v) \geq p{m - 1 \choose r - 1}$. Write $f(G)$ for the order of a largest full subgraph of $G$ and $f^ r(n,p)$ for the minimum of $f(G)$ over all $r$-uniform hypergraphs $G$ on $n$ vertices with density $p$. 
The problems
studied in this paper can be generalised to uniform
hypergraphs, and in particular, one may ask for the order of $f^r(n,p)$.
It is straightforward to imitate the proof of the upper bound in Theorem~\ref{main} to show that for certain values of $p \in (0,1)$
that $f^r(n,p) = O(n^{\sfrac{r}{r + 1}})$. We leave open
the problem of determining the tightness of this upper
bound, as well as of giving bounds on $f^r(n,p)$ when $p$
decays as $n$ grows:
 \begin{problem}
 Determine the order of magnitude of $f^r(n,p)$ for $p=p(n)$ and $r \geq 3$.
 \end{problem}
 The greedy algorithm from Section~\ref{proofofgreedy} generalises to $r$-uniform hypergraphs, in which it yields a
 full subgraph of order only $n^{\sfrac{1}{r}}$. It is an open question as to how the results on relatively $q$-full subgraphs in
 Section~\ref{section: qfull} may be extended to $r$-uniform hypergraphs.
 On the other hand, Proposition~\ref{theorem: Thomason} was
 generalised to hypergraphs by Haviland and
Thomason~\cite{HavilandThomason89}, and using their
result, the proof of the lower bound in Theorem~\ref{theorem:
g(n)=Omega(n/log n)} extends to the setting of
$r$-uniform hypergraphs exactly as before. In particular,
every $n$-vertex $r$-uniform hypergraph contains a full
or a co-full subgraph with $\Omega(n/\log n)$ vertices.


$\bullet$ {\bf Digraphs.} We could also ask about
directed graphs. Since every subgraph of a transitive
tournament has a vertex of in-degree zero and a vertex of
out-degree zero, it is more fruitful to ask about
extensions of Theorem~\ref{relative} than of
Theorem~\ref{main}. Let $d_H^+(v)$ denote the out-degree
of a vertex $v\in H$. A subgraph $H$ of a directed graph
$D$ is \emph{relatively $q$-out-full} if for every $v\in
H$ we have $d_H^+(v)\geq q d_D^+(v)$. Then the problem is
to determine the smallest function $h(n,q)$ such that
every digraph on $n$ vertices has a relatively
$q$-out-full subgraph with at most $h(n,q)$ vertices.


$\bullet$ {\bf Weighted graphs.} A \emph{weighted graph} is a pair $W=(V,w)$, where $w: \ V^{(2)}\rightarrow [0,1]$ is a weighting of pairs of vertices from $V$. The \emph{density} $p$ of $W$ is then the average pair-weight under $w$, and the \emph{degree} $d_Y(x)$ of $x$ in a subset $Y\subseteq V$ is the sum over all $y \in Y$ of $w(\{x,y\})$. Our definition of full subgraphs carries over to the weighted graph setting in the natural way, and we can ask:
\begin{problem}\label{question: weighted rich}
Let $W$ be a weighted graph. Determine tight lower bounds for the order of a largest full subgraph of $W$.
\end{problem}
We may similarly ask about relatively $q$-full subgraphs for weighted graphs.  The proof of Theorem~\ref{q-full or (1-q)-full} extends naturally to this setting: running through the same argument as before, we obtain a new equality
\[M'=M + (1-q)d(Y)-d_Y(y) +qd(x)-d_X(x) -w(\{x,y\}) .\]
From the maximality of $M=(1-q)e(X)+qe(Y)$ over all weighted bipartitions, we deduce, in replacement of conditions (a)--(c), that 
\[ qd(x)- d_X(x) \leq w(\{x,y\})\qquad \textrm{ and } \qquad (1-q)d(y)-d_Y(y)\leq w(\{x,y\}). \] 
This implies as before that $X\cup\{y\}$ is relatively $q$-full on $\lceil qn \rceil +1$ vertices and $Y\cup\{x\}$ is relatively $(1-q)$-full on $\lfloor(1-q)n\rfloor+1$ vertices. Weighted analogues of Theorems~\ref{relative} and~\ref{1/r-full} follow as immediate corollaries. However, this does not quite allow us to prove a weighted version of Theorem~\ref{main}: in the weighted setting, degrees need not be integers and so the minimum degree lower bound (\ref{inequality: mindeg in relatively 2^{-t}-full subgraph})
becomes $D>q$, while the upper bound on the degree requirement for being  a full subgraph in (\ref{inequality: requirement for full}) remains $q+1-p2^{-t}$. Thus our argument in this case only yields a `weakly' full subgraph on $N$ vertices (i.e. with minimum degree greater than $p(N-1)-1$, rather than $p(N-1)$). Is this the best that can be done? We leave this as an open problem.

We note that the proofs of Theorems~\ref{greedy} and \ref{theorem: f(G) geq disc+/alpha} carry over to the weighted setting without any changes. To recover a weighted analogue of our result on full and co-full subgraphs, Theorem~\ref{theorem: g(n)=Omega(n/log n)}, a weighted version of Thomason's theorem (Proposition~\ref{theorem: Thomason}) would be needed. This problem too is left open.



$\bullet$ {\bf Computational complexity.} It appears likely that the following computational problems, which we have not investigated,
are of similar complexity to Max Cut:
\begin{enumerate}[(i)]
\setlength\itemsep{-0.3em}
\item find a largest full subgraph of $G$;
\item given an integer $k$, determine whether or not $G$ contains a full subgraph on $k$ vertices;
\item given an integer $k$, find a $k$-vertex subgraph with largest minimum degree.
\end{enumerate}
The problem of finding (an approximation to) the densest
subgraph of order $k$ has received a significant amount
of attention from the computer science community (see
e.g.
\cite{BhaskaracharikaChlamtacFeigeVijayaraghavan10,FeigePelegKortsarz01}),
including in some variants involving degree
constraints~\cite{AminiPelegPerennesSauSaurabh09}. From
an algorithmic perspective, we thus expect that problems
(i)--(iii) above will be hard: densest subgraph of order
$k$ is known to be NP-Hard. Further, examples due to
Sch\"affer and Yannakakis~\cite{SchafferYannakakis91} and
Monien and Tscheuschner~\cite{MonienTscheuschner10} for
weighted versions of the Max-Cut problem suggest that
local search (i.e. algorithms based on flipping vertices
between a $k$-set $X$ whose minimum degree we are trying
to maximise and its complement) could take exponential
time to converge to a local optimum for problem (iii)
(see also the work of Poljak~\cite{Poljak95}). We note
that the proofs of Theorems~\ref{main}, \ref{greedy},
\ref{relative} and \ref{1/r-full} yield polynomial time algorithms in each
case.

\section*{Acknowledgements}
We are very grateful to the anonymous referees for pointing out important mistakes in our initial proof of Theorem~\ref{theorem: g(n)=Omega(n/log n)}, and for further making several useful suggestions which helped improve the presentation and correctness of the paper.

This research was conducted while the authors were attending the research semester on Graphs, Hypergraphs and Computing at the Institut Mittag-Leffler (Djursholm, Sweden), whose hospitality the authors gratefully acknowledge. We also thank L. Lov\'{a}sz who pointed out the work of S. Poljak, and A. Scott for a helpful discussion of counterexamples for partition problems.


\end{document}